\documentclass[12pt]{amsart} 
\usepackage[english]{babel}
\usepackage[utf8]{inputenc}
\usepackage{amsfonts}
\usepackage[version=3]{mhchem}
\usepackage{amsthm}
\usepackage{amssymb}
\usepackage{mathtools}
\usepackage[top=1in, bottom=1in, left=1in, right=1in]{geometry}
\usepackage{array}
\usepackage{tabu}
\usepackage{caption} 
\usepackage[labelfont=bf]{caption}
\usepackage{float}
\usepackage{tikz}
\usetikzlibrary{arrows, positioning, quotes, shapes}
\usetikzlibrary{arrows.meta}
\usepackage{ dsfont }
 
\usepackage[foot]{amsaddr}

\newcolumntype{M}[1]{>{\centering\arraybackslash}m{#1}}

\title[Identifiability of linear compartmental models]{Identifiability of linear compartmental models: \\ the impact of removing leaks and edges}

\author[P.~Chan]{Patrick Chan$^1$}
\address{$^1$Loyola University Chicago}

\author[K.~Johnston]{Katherine Johnston$^2$}
\address{$^2$Lafayette College}

\author[A.~Shiu]{Anne Shiu$^3$}
\address{$^3$Texas A\&M University}

\author[A.~Sobieska]{Aleksandra Sobieska$^4$}
\address{$^4$University of Wisconsin-Madison}

\author[C.~Spinner]{Clare Spinner$^5$}
\address{$^5$University of Portland}

\date{June 21,  2021}

\newtheorem{theorem}{Theorem}[section]

\newtheorem{prop}[theorem]{Proposition}
\newtheorem{proposition}[theorem]{Proposition}

\newtheorem{lemma}[theorem]{Lemma}
\newtheorem{conjecture}[theorem]{Conjecture}

\theoremstyle{definition}
\newtheorem{definition}[theorem]{Definition}

\newtheorem{example}[theorem]{Example}
\newtheorem{remark}[theorem]{Remark}
\newtheorem{question}[theorem]{Question}
\newtheorem{assumption}[theorem]{Assumption}

\begin{document}

\maketitle
\renewenvironment{abstract}
 {\quotation\small\noindent\rule{\linewidth}{.5pt}\par\smallskip
  {\centering\bfseries\abstractname\par}\medskip}
 {\par\noindent\rule{\linewidth}{.5pt}\endquotation}

\begin{abstract}
A mathematical model is identifiable if its  parameters can be recovered from data. Here, we focus on a particular class of model, linear compartmental models, which are used to represent the transfer of substances in a system. 
We analyze what happens to identifiability when operations are performed on a model, specifically, adding or deleting a leak or an edge. 
We first consider the conjecture of Gross {\em et al.}\ 
that states that removing a leak from an identifiable model yields a model that is again identifiable.  We prove a special case of this conjecture, and also show that the conjecture is equivalent to asserting that leak terms do not divide the so-called singular-locus equation.  
As for edge terms that do divide this equation, we conjecture that removing any one of these edges makes the model become unidentifiable, and then prove a case of this somewhat surprising conjecture.  
\end{abstract}

\section{Introduction} \label{sec:intro}
A model is \textit{generically structurally identifiable} 
if it has the following desirable property: from generic values of the 
inputs and initial conditions, the model parameters can be recovered 
from exact measurements of both inputs and outputs~\cite{Bellman}. 
In this paper, we consider the problem of assessing structural identifiability for a particular class of models, called \emph{linear compartmental models}. Linear compartmental models are used in a variety of fields, including in biological applications, to represent the transfer of substances in a system (e.g., \cite{dargenio1988simulation, hori2006role, mulholland1974analysis}). 
In the context of pharmacokinetics, a better understanding of such models might 
reveal better drug dosing regimes and more precise measurements of drug transfers within the body. 

We focus on the question of how identifiability is affected when operations -- such as adding or removing inputs, outputs, leaks, or edges -- are performed on a model~\cite{bortner-meshkat, joining, GHMS19One, vajda1984, vajdaetal}.   Gross {\em et al.}\ proved that 
under certain hypotheses, adding or removing a leak preserves identifiability~\cite{GHMS19One}. 
Subsequently, Gerberding {\em et al.}\ 
proved related results 
-- pertaining to adding or removing leaks and edges, and moving the output -- 
for several common families of models, specifically, catenary, cycle, and mammillary models~\cite{gerberding2019identifiability}.
In this work, we consider more general models, and again investigate the effect of adding or removing a leak or edge. 

Our two guiding questions, which were posed in earlier work~\cite{GHMS19One, GNA17Three}, are as follows. First, does deleting a leak from an identifiable model yield a model that is also identifiable? 
The second question pertains to the {\em singular-locus equation}, which defines where the
Jacobian matrix of a model's coefficient map (which is used for analyzing identifiability)  
drops rank.  
It is known that deleting an edge whose corresponding parameter does \underline{not} divide the singular-locus equation yields a submodel that is again identifiable~\cite{GNA17Three}.  The question here is, for an edge that \underline{does} divide the singular-locus equation, is the resulting submodel unidentifiable?  An answer would help clarify what information the singular-locus equation provides regarding the identifiability of  submodels~\cite{GNA17Three}. 

Both of the aforementioned questions 
are conjectured to have affirmative answers; see Conjecture~\ref{conj:rmv-leak}, which was originally posed in~\cite{GHMS19One}, and Conjecture~\ref{conj:dividing-edge} (see also the summary in Table~\ref{tab:summary-results}).  
We prove several results toward these conjectures. First, we show a special case of an equivalent conjecture on adding leaks, when unidentifiability arises from having too many parameters 
(Theorem~\ref{thm:add-leak}). We also show that the conjecture 
on deleting leaks (Conjecture~\ref{conj:rmv-leak})
is equivalent to asserting that leak terms do not divide the singular-locus equation (Theorem~\ref{thm:equivalence-of-conjectures}).  
Next, we prove that the conjecture
on dividing edges (Conjecture~\ref{conj:dividing-edge})
holds for models that satisfy certain combinatorial conditions (Theorem~\ref{thm:dividing-edge}).  
Finally, we investigate 
the types of edges that divide the singular-locus equation.

\begin{table}[ht]
\begin{tabular}{l c c}
    \hline
     Operation & Conjectured effect & Result \\
     \hline
     Remove leak & Identifiable & Theorem~\ref{thm:add-leak} \\
     Remove dividing edge & Unidentifiable & Theorem~\ref{thm:dividing-edge} \\
     \hline
\end{tabular}
    \caption{Summary of conjectures and results on how operations affect identifiability.  
    For a strongly connected linear compartmental model $\mathcal{\mathcal{M}}$ with at least one input, if $\widetilde{\mathcal{M}}$ is obtained from $\mathcal{\mathcal{M}}$ by the specified operation, and $\mathcal{\mathcal{M}}$ is identifiable, then $\widetilde{\mathcal{M}}$ is conjectured to be 
    as listed in the second column.  Partial results toward these conjectures are listed in the last column.  For related prior results, 
    see~\cite[Table~1]{gerberding2019identifiability} 
    and~\cite[Tables~1 and~2]{GHMS19One}. } \label{tab:summary-results}
\end{table}

Our investigations into identifiability harness the theory of input-output equations for linear compartmental models~\cite{bearup,  glad,  Ovchinnikov-Pogudin-Thompson}.
Our proofs therefore use linear-algebraic and combinatorial arguments.
We also rely on results of 
Bortner, Gross, Meshkat, Shiu, and Sullivant~\cite{bortner}, 
Bortner and Meshkat~\cite{bortner-meshkat}, and Mehskat, Sullivant, and Eisenberg~\cite{MSE-iden-results-several}.

The outline of this paper is as follows. In Section~\ref{sec: background}, we introduce linear compartmental models and identifiability.  We prove our results pertaining to leaks and dividing edges in Sections~\ref{sec:results-leak} and~\ref{sec:edge}, respectively.  We conclude with a discussion in Section~\ref{sec:discussion}.

\section{Background} \label{sec: background}
This section introduces linear compartmental models, identifiability, and the singular locus.
We closely follow the notation in~\cite{gerberding2019identifiability, GHMS19One, GNA17Three}.

\subsection{Linear compartmental models}

A \emph{linear compartmental model} $(G, In, Out, Leak)$ 
consists of a directed graph $G = (V,E)$ and subsets $In, Out, Leak \subseteq V$ of (respectively) {\em inputs, outputs,} and {\em leaks}. 
Each $i \in V$ is a {\em compartment.}  Input compartments receive an input stream, $u_i(t)$; and output compartments have an associated output measure, $y(t)$. 

\begin{assumption} \label{assm:at-least-1-output}
Throughout this work, we assume 
that every linear compartmental model has at least one output, because models without outputs are not identifiable.
\end{assumption}

In keeping with the literature, output compartments are indicated using this symbol: 
\begin{tikzpicture}
    \draw (0,0) circle (0.06);
    \draw (0.06,0.06) -- (.16,.19);
\end{tikzpicture}.
Input compartments are labeled ``in", and leaks are indicated by outgoing edges. For example, the linear compartmental model in Figure 1 has $In = \{1\}$, $Out = \{3\}$, and $Leak = \{1,2\}$. 

\begin{figure}[h]
    \centering
    \begin{tikzpicture}
    \draw (0,0) circle (.5) node [text=black] {$1$}; 
    \draw (2,0) circle (.5) node [text=black] {$2$}; 
    \draw (4,0) circle (.5) node [text=black] {$3$}; 
    \draw (5,.75) circle (0.1); 
    \draw (4.45,.25) -- (4.94,.7); 
    \draw (.5,.25) -- (1.5,.25) [-{Latex[length=2mm]}] node [label={[xshift=-.475cm, yshift=-.15cm]{\footnotesize $k_{21}$}}]{}; 
    \draw (.5,-.25) -- (1.5,-.25) [{Latex[length=2mm]}-] node [label={[xshift=-.475cm, yshift=-.75cm]{\footnotesize $k_{12}$}}]{}; 
    \draw (2.5,.25) -- (3.5,.25) [-{Latex[length=2mm]}] node [label={[xshift=-.475cm, yshift=-.15cm]{\footnotesize $k_{32}$}}]{}; 
    \draw (2.5,-.25) -- (3.5,-.25) [{Latex[length=2mm]}-] node [label={[xshift=-.475cm, yshift=-.75cm]{\footnotesize $k_{23}$}}]{}; 
    \draw (-0.5,.25) -- (-1, .75) [{Latex[length=2mm]}-]; 
    \node at (-1, .35) {in}; 
    \draw (0.2,.5) -- (.8,1.1) [-{Latex[length=2mm]}] node [label={[xshift=-.5cm, yshift=-.5cm]{\footnotesize $k_{01}$}}]{}; 
    \draw (2.2,.5) -- (2.8,1.1) [-{Latex[length=2mm]}] node [label={[xshift=-.5cm, yshift=-.5cm]{\footnotesize $k_{02}$}}]{}; 
\end{tikzpicture}
    \caption{A linear compartmental model.}
    \label{fig:1}
\end{figure}
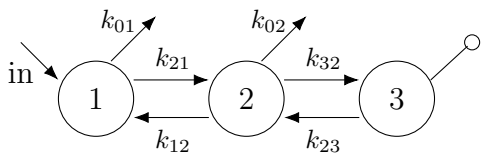

An edge $j \to i \in E$ represents flow from compartment $j$ to compartment $i$. We assign a \emph{flow parameter} $k_{ij}$ to each such edge. Leak compartments also have flow parameters, $k_{0 \ell}$, which are the rates of flow exiting the system from compartment $\ell$. 

Figure~\ref{fig:1} depicts a three-compartment catenary model. From a biological standpoint, this model could represent the injection and flow of a drug within the body. The input is the drug, and compartment 1 is the injection site. Compartments 2 and 3 represent other organs in the body where the drug travels. Edges between compartments 1, 2, and 3 indicate the transfer of the drug between organs, and the leaks, labeled by parameters $k_{01}$ and $k_{02}$, represent the transfer of drugs from the measurable system into immeasurable parts of the body, such as the bloodstream. The output is where the concentration of the drug is measured.

We now introduce some more definitions.

\begin{definition}
A directed graph is \emph{strongly connected} if there exists a path from each vertex to every other vertex. A linear compartmental model ($G, In, Out, Leak$) is \emph{strongly connected} if $G$ is strongly connected.
\end{definition}

\begin{definition}
For a linear compartmental model $(G, In, Out, Leak)$ with $n$ compartments, the \emph{compartmental matrix} is the $n \times n$ matrix $A$ given by the following: \par 
\begin{center}
$A_{ij} \coloneqq \begin{cases}
		-k_{0i} - \sum\limits_{p:(i \to p) \in E} k_{pi} & \text{if $i = j$ and $i \in Leak$} \\
		- \sum\limits_{p: (i \to p) \in E} k_{pi} & \text{if $i = j$ and $i \notin Leak$} \\
		k_{ij} & \text{if $j \to i$ is an edge of $G$} \\
		0 & \text{otherwise}
		\end{cases} $
\end{center}
\end{definition}

A linear compartmental model $(G, In, Out, Leak)$ defines the following system of ODEs with inputs $u_i(t)$ and outputs $y_i(t)$, where $x(t) = (x_1(t), x_2(t), \dots, x_n(t))$ is the vector of concentrations in the compartments at time $t$: 
\begin{align} \label{eq:ode}
    x'(t) &= Ax(t) + u(t), \\
    y_i(t) &= x_i(t) \quad \quad \text{for all $i \in Out$}~, \notag
\end{align}
where $u_i(t) \equiv 0$ for $i \notin In$.

\begin{example}
For the model in Figure 1, the ODEs are given by 

\[
  \begin{pmatrix}
  x'_1 \\
  x'_2 \\
  x'_3 \\
\end{pmatrix} 
=
 \begin{pmatrix}
 -k_{01}-k_{21} & k_{12} & 0 \\
 k_{21} & -k_{02}-k_{12}-k_{32} & k_{23} \\
 0 & k_{32} & -k_{23} \\
 \end{pmatrix}
\begin{pmatrix}
x_1 \\
x_2 \\
x_3 \\
\end{pmatrix}
+
\begin{pmatrix}
u_1 \\
0 \\
0 \\
\end{pmatrix}~,
\]
with output equation $y_3 = x_3$.
\end{example}

\subsection{Input-Output Equations}
\emph{Input-output equations} of a linear compartmental model are equations that hold along all solutions of the ODEs~\eqref{eq:ode}, and involve only the parameters~$k_{ij}$, input variables $u_i$, output variables $y_i$, and their derivatives. A general form of such equations was proven by Meshkat, Sullivant, and Eisenberg~\cite[Theorem~2]{MSE-iden-results-several} (see also~\cite[Remark~2.7]{GHMS19One}), as follows. 

\begin{prop} 
Consider a linear compartmental model $\mathcal M$ that is strongly connected and has at least one input.
Let $A$ denote the compartmental matrix, let $\partial$ be the differential operator $d/dt$, and let $(\partial I - A)_{ij}$ be the submatrix of $(\partial I - A)$ obtained by removing row $i$ and column~$j$. Then, for $j \in Out$, an input-output equation of $\mathcal M$ involving $y_j$ is:
\begin{align}
    \label{eq:i-o-equation}
    \det(\partial I - A) y_j ~=~ \sum_{i \in In}  (-1)^{i+j}  \det \left( \partial I-A \right)_{ij} u_i ~.
\end{align} 
\end{prop}

We view the input-output equations~\eqref{eq:i-o-equation} as polynomials in the $y_j$'s and $u_i$'s and their derivatives.  The coefficients of the equations are therefore polynomials in the parameters ($k_{\ell m}$ for edges $m \to \ell$, and $k_{0 p}$ for leaks $p \in Leak$).  The coefficients of the input-output equations~\eqref{eq:i-o-equation} that are non-constant polynomials
%
are called the {\em nontrivial coefficients}.

\begin{definition} \label{def:coeff-map}
For a linear compartmental model $\mathcal M$, 
the \emph{coefficient map} $c: \mathbb{R}^{|E|+|Leak|} \to \mathbb{R}^m$ sends the vector of parameters 
$(k_{ij})$ to the vector of
all nontrivial 
coefficients of
all input-ouput equations~\eqref{eq:i-o-equation}. Here, $m$ denotes the number of such coefficients. 
\end{definition}

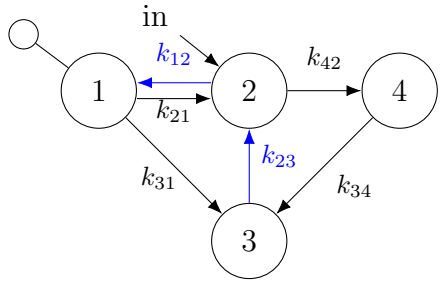
\begin{figure}[ht]
    \centering
    \begin{tikzpicture}
        \node[circle,draw, minimum size=1cm] (1) at  (0,0)  {1};
        \node[circle,draw, minimum size=1cm] (2) at  (2,0)  {2};
        \node[circle,draw, minimum size=1cm] (4) at  (4,0)  {4};
        \node[circle,draw, minimum size=1cm] (3) at  (2,-2)  {3};
        
        \node[circle, minimum size=0cm] (in) at  (.75, 1)  {in};
        \draw (in) -- (2) [-{Latex[length=2mm]}];
        \node[circle, draw, minimum size=.01cm] (out) at  (-1,.75)  {};
        \draw (out) -- (1);
        
        \draw (1) -- (2) [-{Latex[length=2mm]}, transform canvas={yshift=-3pt}] node [label={[xshift=-1cm, yshift=-.6cm]{\footnotesize $k_{21}$}}]{}; 
        \draw [blue] (2) -- (1) [-{Latex[length=2mm]}, transform canvas={yshift=3pt}] node [label={[xshift=1cm, yshift=-.05cm]{\color{blue} \footnotesize $k_{12}$}}]{}; 
        \draw (2) -- (4) [-{Latex[length=2mm]}] node [label={[xshift=-1cm, yshift=0cm]{\color{black} \footnotesize $k_{42}$}}]{}; 
        \draw (3) -- (2) [-{Latex[length=2mm]}, color = blue] node [label={[xshift=.4cm, yshift=-1.3cm]{\color{blue} \footnotesize $k_{23}$}}]{}; 
        \draw (1) -- (3) [-{Latex[length=2mm]}] node [label={[xshift=-1.2cm, yshift=.4cm]{\footnotesize $k_{31}$}}]{}; 
        \draw (4) -- (3) [-{Latex[length=2mm]}] node [label={[xshift=1.4cm, yshift=.3cm]{\footnotesize $k_{34}$}}]{}; 
    \end{tikzpicture}
    \caption{A linear compartmental model with $In = \{2\}$, $Out = \{1\}$, and $Leak=\emptyset$.  The dividing edges, $k_{12}$ and $k_{23}$,  are shown in blue (see Section~\ref{sec:singular-locus}).  }
    \label{fig:appearing1}
\end{figure}

\begin{example} \label{ex:running-example}
For the model in Figure~\ref{fig:appearing1}, the compartmental matrix is: 
\begin{align*}
    A ~=~
    \begin{pmatrix}
    -k_{21} - k_{31}    & k_{12}            & 0         & 0 \\    
    k_{21}              & -k_{12}-k_{42}    & k_{23}    & 0 \\
    k_{31}              & 0                 & -k_{23}   & k_{34} \\
    0                   & k_{42}            & 0         & -k_{34}
    \end{pmatrix}~.
\end{align*}
The input-output equation~\eqref{eq:i-o-equation}, namely, $\det(\partial I - A) y_1 = - \det \left( \partial I-A \right)_{21} u_2$, is as follows:
    \begin{align*}
    \det 
    \begin{pmatrix}
    \partial + k_{21} + k_{31}    & -k_{12}            & 0         & 0 \\    
    -k_{21}              & \partial + k_{12} + k_{42}    & -k_{23}    & 0 \\
    -k_{31}              & 0                 & \partial + k_{23}   & -k_{34} \\
    0                   & -k_{42}            & 0         & \partial + k_{34}
    \end{pmatrix} 
    y_1 
    ~=~
    \det 
    \begin{pmatrix}
    -k_{12}            & 0         & 0 \\    
     0                 & \partial + k_{23}   & -k_{34} \\
    -k_{42}            & 0         & \partial + k_{34}
    \end{pmatrix} 
    u_2~,
    \end{align*}
which expands to:
\begin{align} \label{eq:ex-i-o-eqn}
    y_1^{(4)}
    & 
    +
    (k_{12} + k_{21} + k_{23} + k_{31}+ k_{34} + k_{42}) y_1^{(3)}
    \\
    \notag
    & 
    + 
    (k_{12}k_{23} + k_{21}k_{23} + k_{12}k_{31} + k_{23}k_{31} + k_{12}k_{34} + k_{21}k_{34} + k_{23}k_{34} + k_{31}k_{34} 
    \\
    \notag
    & \quad \quad 
    + k_{21}k_{42} + k_{23}k_{42} + k_{31}k_{42} + k_{34}k_{42}) y_1^{(2)}
    \\
    \notag
    &
    + 
    (k_{12}k_{23}k_{34} + k_{21}k_{23}k_{34} + k_{12}k_{31}k_{34} + k_{23}k_{31}k_{34} + k_{21}k_{23}k_{42} 
    \\
    \notag
    & \quad \quad 
    + k_{23}k_{31}k_{42} + k_{21}k_{34}k_{42} + k_{31}k_{34}k_{42})y_1'
    \\
    \notag
    ~=~ &
    k_{12} u_2^{(2)} +
    k_{12}(k_{23}+k_{34})u_2' +
    k_{12}k_{23}k_{34} u_2~.
\end{align}
The nontrivial coefficients in the input-output equation~\eqref{eq:ex-i-o-eqn} form the coefficient map $c:\mathbb{R}^6 \to \mathbb{R}^6$; for instance, $c_1=k_{12} + k_{21} + k_{23} + k_{31}+ k_{34} + k_{42}$.
\end{example}

The following result, 
which is due to Bortner {\em et al.}~\cite[Corollary~3.4]{bortner}, addresses the size (the value of $m$) of the coefficient map $c:~\mathbb{R}^{|E|+|Leak|}~\to~\mathbb{R}^m$ for models with only one input and one output.

\begin{proposition}[Number of coefficients~\cite{bortner}] \label{prop:number-of-coefficients}
Consider a strongly connected linear compartmental model $ (G, In, Out, Leak)$ with $|In| = |Out| = 1$. Let $n$ be the number of compartments, and let $L$ be the length of the shortest (directed) path in $G$ from the input compartment to the output compartment. Then, in the input-output equation \eqref{eq:i-o-equation}, the number of 
nontrivial
coefficients on the left-hand and right-hand sides are as follows: \medskip \newline 
\begin{minipage}[t]{0.5 \textwidth}
$\#$ on LHS = $\begin{cases}
$n$ & \text{if $Leak \neq \emptyset$} \\
$n-1$ & \text{if $Leak = \emptyset$}\\
\end{cases}$
\end{minipage}
\begin{minipage}[t]{.5\textwidth}
$\#$ on RHS = $\begin{cases}
$n-1$ & \text{if $In = Out$}\\
$n-L$ & \text{if $In \neq Out$}\\
\end{cases}$
\end{minipage}
\end{proposition}
%

In this work, we are interested in what happens to identifiability when leaks or edges are added or removed.  The following lemma analyzes the effect on the coefficients.

\begin{lemma}[Coefficients and adding a leak or edge] \label{lem:coeff-add-leak}
Let $\mathcal{M} = (G, In, Out, Leak)$ be a strongly connected linear compartmental model with $In=\{i\}$ and $Out=\{j\}$.
Let $n$ denote the number of compartments. 
Let $\widetilde{\mathcal{M}}$ denote the model obtained from $\mathcal M$ obtained by adding a leak or an edge, with corresponding parameter $k_{uv}$.  
Let $d_0, d_1, \dots, d_{2n+1}$ denote the coefficient of (respectively) 
    $y_j^{(0 )}, y_j^{(1 )}, \dots, y_j^{(n )},~ u_i^{(0 )}, u_i^{(1 )}, \dots, u_i^{(n-1 )} $ in the input-output equation~\eqref{eq:i-o-equation} for $\mathcal{M}$.
Let $\widetilde{d_0}, \widetilde{d_1}, \dots, \widetilde{d}_{2n+1}$ denote the 
corresponding coefficients (respectively) for $\widetilde{\mathcal{M}}$.
Then, for $\ell = 0,1,\dots, 2n+1$,
\[
    \widetilde{ d_{\ell} } ~=~ d_{\ell} + k_{uv} \cdot g_{\ell}~,
\]
for some polynomial $g_{\ell}$ (with integer coefficients) in the flow parameters of $\mathcal{M}$.
\end{lemma}
\begin{proof}
Let $A$ and $\widetilde{A}$ denote the compartmental matrices of $\mathcal{M}$ and $\widetilde{\mathcal{M}}$, respectively.  Note that $A= ( \widetilde{A})|_{k_{uv}=0}$.  Now the result is immediate from equation~\eqref{eq:i-o-equation}.
\end{proof}

\subsection{Identifiability}
A linear compartmental model is \emph{structurally identifiable} if each of the parameters $k_{ij}$ can be recovered from data~\cite{Bellman}.  (In contrast, a model is \emph{practically identifiable} if the parameters can be recovered from data with noise.)  Our focus here is on structural identifiability.  More precisely, we are interested in
generic local identifiability, which 
refers to identifiability except possibly for a measure-zero set of parameter space
and also allows for 
recovery of the parameters up to a finite set. 
This concept, in the case of strongly connected models (and others as well), is captured by the following definition in terms of input-output equations~\cite[Corollary 3.2]{Ovchinnikov-Pogudin-Thompson}. 

\begin{definition} \label{def:identifiable}
Let $\mathcal{M} = (G, In, Out, Leak)$ be a strongly connected linear compartmental model with 
at least one input. 
Let $c : \mathbb{R}^{|E| + |Leak|} \to \mathbb{R}^m$ be the coefficient map derived from the input-output equations~\eqref{eq:i-o-equation}. Then, $\mathcal{M}$ is: 

\begin{enumerate}
    \item \emph{generically locally identifiable} if, outside a set of measure zero, every point in $\mathbb{R}^{|E|+|Leak|}$ has an open neighborhood $U$
    for which the restriction $c|_U : U \to \mathbb{R}^m$ is one-to-one; 
    \item \emph{unidentifiable} if $c$ is infinite-to-one.
\end{enumerate}
\end{definition}

Next, we recall the following useful criterion for (generic local) identifiability~\cite{MSE-iden-results-several}.

\begin{prop}[Meshkat, Sullivant, and Eisenberg \cite{MSE-iden-results-several}] \label{prop:MSE}
A linear compartmental model $(G, In, Out, Leak)$, with $G = (V,E)$, is generically locally identifiable if and only if the rank of the Jacobian matrix of its coefficient map, when evaluated at a generic point, is equal to $|E|+|Leak|$.
\end{prop}

\begin{example}[Example~\ref{ex:running-example}, continued] \label{ex:running-example-2}
Continuing with the model in Figure~\ref{fig:appearing1}, the Jacobian matrix of the coefficient map (arising from the input-output equation~\eqref{eq:ex-i-o-eqn}) is: 

\[
    \hspace{-2.4cm}
    \small{
    \left( 
    \setlength{\tabcolsep}{0.65em}
    {\renewcommand{\arraystretch}{1.3}
    \begin{tabular}{M{2.5cm}|M{3cm}|M{3cm}|M{3cm}|M{3cm}|M{3cm}}
    1 & 1 & 1 & 1 & 1 & 1 \\
    \hline 
    $k_{23} + k_{31} + k_{34}$ & $k_{23} + k_{34} + k_{42}$ & 
    $k_{12} + k_{21} + k_{31} + k_{34} + k_{42}$ & $k_{12} + k_{23} + k_{34} + k_{42}$ & 
    $k_{12} + k_{21} + k_{23} + k_{31} + k_{42}$ & $k_{21} + k_{23} + k_{31} + k_{34}$\\
    \hline 
    $k_{23}k_{34} + k_{31}k_{34}$ & $k_{23}k_{34} + k_{23}k_{42} + k_{34}k_{42}$ & 
    $k_{12}k_{34} + k_{21}k_{34} + k_{31}k_{34} + k_{21}k_{42} + k_{31}k_{42}$ & $k_{12}k_{34} + k_{23}k_{34} + k_{23}k_{42} + k_{34}k_{42}$ & 
    $k_{12}k_{23} + k_{21}k_{23} + k_{12}k_{31} + k_{23}k_{31} + k_{21}k_{42} + k_{31}k_{42}$ & $k_{21}k_{23} + k_{23}k_{31} + k_{21}k_{34} + k_{31}k_{34}$\\
    \hline 
    1 & 0 & 0 & 0 & 0 & 0\\
    \hline 
    $k_{23}+k_{34}$ & 0 & $k_{12}$ & 0 & $k_{12}$ & 0 \\
    \hline 
    $k_{23}k_{34}$ & 0 & $k_{12}k_{34}$ & 0 & $k_{12}k_{23}$ & 0\\ 
    \end{tabular}
    }
    \right)
    }
\] 
The determinant of this matrix is the following nonzero polynomial (and so, for generic values of the $k_{ij}$, the matrix has rank 5):
\begin{align} \label{eq:ex-s-locus-eqn}
k_{12}^3 (k_{21} + k_{31} - k_{34} - k_{42}) (k_{23} - k_{34}) k_{23}~.
\end{align}
It therefore follows from Proposition~\ref{prop:MSE} that this model is generically locally identifiable.
\end{example}

An important open problem is how to discern identifiability directly from a model.  This question is challenging and subtle, as the models in 
Figures~\ref{fig:Mysterious1} and~\ref{fig:Mysterious2} show.  The two models are very similar -- the second model is obtained from the first by replacing the edge $k_{41}$ with $k_{31}$ -- and yet the first model is identifiable while the second is unidentifiable.

\begin{minipage}{0.5 \textwidth}
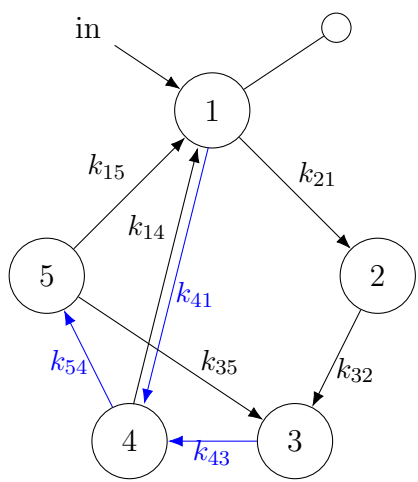
\begin{figure}[H]
    \centering
    \begin{tikzpicture}[scale=1.1]
        \node[circle,draw, minimum size=1cm] (5) at  (0,0) {5}; 
        \node[circle,draw, minimum size=1cm] (1) at  (2,2)  {1};
        \node[circle,draw, minimum size=1cm] (2) at  (4,0)  {2};
        \node[circle,draw, minimum size=1cm] (4) at  (1,-2)  {4};
        \node[circle,draw, minimum size=1cm] (3) at  (3,-2)  {3};
        \draw (1) -- (2) [-{Latex[length=2mm]}] node [label={[xshift= -.8cm, yshift= .9cm]{\small $k_{21}$}}]{};
        \draw (2) -- (3) [-{Latex[length=2mm]}] node [label={[xshift= .8cm, yshift= .5cm]{\small $k_{32}$}}]{};
        \draw (3) -- (4) [-{Latex[length=2mm]}, color = blue] node [label={[xshift= 1.1cm, yshift= -.6cm]{\color{blue} \small $k_{43}$}}]{};
        \draw (4) -- (5) [-{Latex[length=2mm]}, color = blue] node [label={[xshift= .3cm, yshift= -1.6cm]{\color{blue} \small $k_{54}$}}]{};
        \draw (5) -- (1) [-{Latex[length=2mm]}] node [label={[xshift= -1.4cm, yshift= -1.2cm]{\small $k_{15}$}}]{};
        \draw (1) -- (4) [-{Latex[length=2mm]},transform canvas={xshift=2pt}, color = blue] node [label={[xshift= .8cm, yshift= 1.5cm]{\color{blue} \small $k_{41}$}}]{};
        \draw (5) -- (3) [-{Latex[length=2mm]}] node [label={[xshift= -1cm, yshift= .6cm]{\small $k_{35}$}}]{};
        \draw (4) -- (1) [-{Latex[length=2mm]},transform canvas={xshift=-2pt}] node [label={[xshift= -.8cm, yshift= -2cm]{\small $k_{14}$}}]{};
        \node[circle, minimum size=0cm] (in) at  (.5,3)  {in};
        \draw (in) -- (1) [-{Latex[length=2mm]}];
        \node[circle, draw, minimum size=.01cm] (out) at  (3.5,3)  {};
        \draw (out) -- (1);
    \end{tikzpicture}
        \caption{Model with $In = Out = \{1\}$. The edges $k_{41}$, $k_{43}$, and $k_{54}$, in blue, are dividing edges (see Section~\ref{sec:singular-locus}).}
    \label{fig:Mysterious1}
\end{figure}
\end{minipage}
\begin{minipage}{0.5 \textwidth}
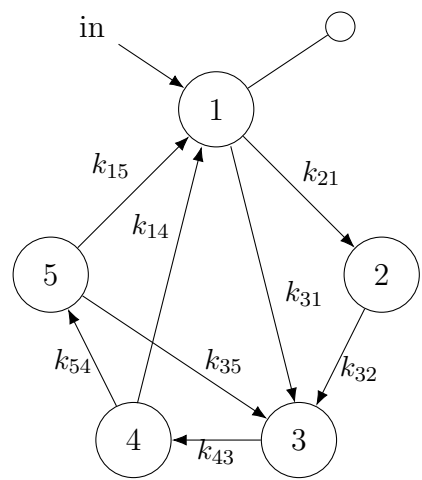
\begin{figure}[H]
    \centering
        \begin{tikzpicture}[scale=1.1]
        \node[circle,draw, minimum size=1cm] (5) at  (0,0) {5}; 
        \node[circle,draw, minimum size=1cm] (1) at  (2,2)  {1};
        \node[circle,draw, minimum size=1cm] (2) at  (4,0)  {2};
        \node[circle,draw, minimum size=1cm] (4) at  (1,-2)  {4};
        \node[circle,draw, minimum size=1cm] (3) at  (3,-2)  {3};
        \draw (1) -- (2) [-{Latex[length=2mm]}] node [label={[xshift= -.8cm, yshift= .9cm]{\small $k_{21}$}}]{};
        \draw (2) -- (3) [-{Latex[length=2mm]}] node [label={[xshift= .8cm, yshift= .5cm]{\small $k_{32}$}}]{};
        \draw (3) -- (4) [-{Latex[length=2mm]}] node [label={[xshift= 1.1cm, yshift= -.6cm]{\small $k_{43}$}}]{};
        \draw (4) -- (5) [-{Latex[length=2mm]}] node [label={[xshift= .3cm, yshift= -1.6cm]{\small $k_{54}$}}]{};
        \draw (5) -- (1) [-{Latex[length=2mm]}] node [label={[xshift= -1.4cm, yshift= -1.2cm]{\small $k_{15}$}}]{};
        \draw (1) -- (3) [-{Latex[length=2mm]},transform canvas={xshift=2pt}] node [label={[xshift= 0cm, yshift= 1.5cm]{\small $k_{31}$}}]{};
        \draw (5) -- (3) [-{Latex[length=2mm]}] node [label={[xshift= -1cm, yshift= .6cm]{\small $k_{35}$}}]{};
        \draw (4) -- (1) [-{Latex[length=2mm]},transform canvas={xshift=-2pt}] node [label={[xshift= -.8cm, yshift= -2cm]{\small $k_{14}$}}]{};
        \node[circle, minimum size=0cm] (in) at  (.5,3)  {in};
        \draw (in) -- (1) [-{Latex[length=2mm]}];
        \node[circle, draw, minimum size=.01cm] (out) at  (3.5,3)  {};
        \draw (out) -- (1);
    \end{tikzpicture}
    \caption{Model with $In = Out = \{1\}$}
    \label{fig:Mysterious2}
\end{figure}
\end{minipage}\medskip \newline

One way to be unidentifiable is to have too many edges or leaks (and thus too many parameters).  The following result in this direction is a special case of a result of Bortner and Meshkat~\cite[Theorem~6.1]{bortner-meshkat}.

\begin{prop}[Unidentifiable from too many leaks] \label{prop:bortner-meshkat-leaks}
Let $\mathcal{M} = (G, In, Out, Leak)$ be a strongly connected linear compartmental model with $|In|=|Out|=1$.  If $|Leak| > |In \cup Out|$, then $\mathcal M$ is unidentifiable.
\end{prop}

One guiding question of our work, which was posed in~\cite[Question~5.2]{GHMS19One}, is as follows.

\begin{question} \label{q:remove-leak}
Let $\widetilde{\mathcal{M}}$ be a linear compartmental model that is generically locally identifiable and has at least one leak.  
If one leak is removed, is the resulting model ${\mathcal{M}}$ always identifiable?
\end{question}

An affirmative answer to Question~\ref{q:remove-leak},
under certain hypotheses,
was conjectured by Gross, Harrington, Meshkat, and Shiu, as follows~\cite[Conjecture 4.5]{GHMS19One}.

\begin{conjecture}[Remove leak] \label{conj:rmv-leak}
Let $\widetilde{\mathcal{M}}$ be a strongly connected linear compartmental model that has at least one input and exactly one leak.  If $\widetilde{\mathcal{M}}$ is generically locally identifiable, then so is the model $\mathcal{M}$ obtained by removing the leak.
\end{conjecture}

Conjecture~\ref{conj:rmv-leak} holds in the following cases:
\begin{itemize}
    \item when $\widetilde{\mathcal{M}}$ has an input, output, and leak in a single compartment (and has no other inputs, outputs, or leaks)~\cite[Proposition~4.6]{GHMS19One}, and
    \item when $\widetilde{\mathcal{M}}$ is obtained from an ``identifiable cycle model''~\cite{MSE-iden-results-several} by removing all leaks except one~\cite[Proposition~3.4]{gerberding2019identifiability}.
\end{itemize}
The first case above is, in fact, a special case of the second~\cite[Corollary~4.2]{bortner-meshkat}.  Also, this second case includes some  models that are widely used in applications, including certain catenary, cycle, and mammillary models~\cite[Proposition~2.11]{GNA17Three}.

\begin{remark} \label{rem:converse-conjecture-remove-leak}
The converse to Conjecture~\ref{conj:rmv-leak} is true~\cite[Theorem 4.3]{GHMS19One}.  
\end{remark}

In the next section, we 
pose the contrapositive of Question~\ref{q:remove-leak}, conjecture an affirmative answer, and then prove one case of this conjecture
(Theorem~\ref{thm:add-leak}). 
In doing so, we resolve a case of (the contrapositive
of) Conjecture~\ref{conj:rmv-leak}.  We also prove that Conjecture~\ref{conj:rmv-leak} is equivalent to a new conjecture, which states that leak terms do not divide the singular-locus equation (see Theorem~\ref{thm:equivalence-of-conjectures}).  We turn now 
to
this topic of the singular locus.




\subsection{Singular Locus} \label{sec:singular-locus}
Here we recall the definition of singular locus introduced in~\cite{GNA17Three}.

\begin{definition}
Let $\mathcal{M} = (G, In, Out, Leak)$ be a 
strongly connected linear compartmental model 
that has at least one input and is
generically locally identifiable. 
Let $c$ be the coefficient map derived from the input-output equations~\eqref{eq:i-o-equation}. 
The \emph{singular locus} is the subset of the parameter space $\mathbb{R}^{|E|+|Leak|}$ where the Jacobian matrix of $c$ has rank strictly less than $|E| + |Leak|$. 
\end{definition}

Therefore, for identifiable linear compartmental models, the {singular locus} is defined by the set of all $(|E| + |Leak|) \times (|E|+|Leak|)$ minors of the Jacobian matrix of the coefficient map. For some models, this Jacobian matrix is square and hence there is only a single minor of interest, namely, the determinant.  In this case, the determinant of the Jacobian matrix is called the {\em singular-locus equation}.  (This equation is also defined in~\cite{GNA17Three} for some non-square cases, but here we mainly focus on the square case.)

We call an edge 
$j \to i \in E$ 
a {\em dividing edge} if its parameter $k_{ij}$ divides the singular-locus equation; we also call $k_{ij}$ itself a dividing edge.  
We do not know whether the dividing edges can be read directly from the model; this is an interesting future research direction. 

\begin{example}[Example~\ref{ex:running-example-2}, continued] \label{ex:running-example-3}
For the model in Figure~\ref{fig:appearing1}, the singular-locus equation was shown in equation~\eqref{eq:ex-s-locus-eqn}.  From this equation, we see that the dividing edges are $k_{12}$ and $k_{23}$.
\end{example}

Gross, Meshkat, and Shiu proved that removing non-dividing edges preserves identifiability, as follows~\cite[Theorem 3.1]{GNA17Three}.

\begin{prop}[Deleting non-dividing edges~\cite{GNA17Three}] \label{prop:delete-non-dividing-edge}
Let $\mathcal{M} = (G, In, Out, Leak)$ be a strongly connected linear compartmental model that is generically locally identifiable, with singular-locus equation $f$. 
Let $\widetilde{\mathcal{M}}$ be the model obtained from $\mathcal{M}$ by deleting a set of edges $\mathcal{I}$ of $G$. If $\widetilde{\mathcal{M}}$ is strongly connected, and 
every edge in $\mathcal I$ is non-dividing, then $\widetilde{\mathcal{M}}$ is generically locally identifiable.
\end{prop}

Gross, Meshkat, and Shiu also asked whether a converse to Proposition~\ref{prop:delete-non-dividing-edge} holds (see~\cite[Question 3.4]{GNA17Three}), and here we conjecture an affirmative answer, as follows.

\begin{conjecture}[Remove dividing edge] \label{conj:dividing-edge}
Let $\mathcal{M}$ be a strongly connected linear compartmental model that is generically locally identifiable, with singular-locus equation $f$.  
If $j \to i$ is a dividing edge of $\mathcal M$, 
and the model $\mathcal{M}'$ 
obtained by deleting the edge $j \to i$
is strongly connected,
then $\mathcal{M}'$
is unidentifiable.
\end{conjecture}

Conjecture~\ref{conj:dividing-edge} is perhaps surprising; indeed, we initially expected a negative answer to~\cite[Question 3.4]{GNA17Three}.  In a later section, we make progress toward Conjecture~\ref{conj:dividing-edge} (see Theorem~\ref{thm:dividing-edge}).

\begin{remark}[Strong-connectedness hypothesis in Conjecture~\ref{conj:dividing-edge}] \label{rem:which-edges-break-strong-connectedness}
In Conjecture~\ref{conj:dividing-edge}, the requirement that the model $\mathcal M'$ be strongly connected is so that (as mentioned earlier) the input-output approach to identifiability is valid~\cite{Ovchinnikov-Pogudin-Thompson}.  Nevertheless, it is our observation that in many cases the dividing edges of a model are such that the removal of such an edge would break strong connectedness. This lack of strong connectedness means that, at least intuitively, information is not fully flowing through the model from input to output, resulting in a situation that would lend itself to being unidentifiable.  This is an interesting avenue for future research.
\end{remark}

\section{Results on leaks} \label{sec:results-leak}
In this section, we prove results on what happens to identifiability when 
leaks are added (Section~\ref{sec:add-leak})
or removed (Section~\ref{sec:remove-leak}).

\subsection{Adding a leak} \label{sec:add-leak}
In this subsection, we address two questions concerning leaks.  
The first, Question~\ref{q:remove-leak}, was posed earlier (see Theorem~\ref{thm:add-leak}).
The second is the question of how many added leaks force an identifiable model to become unidentifiable (see Theorem~\ref{thm:too-many-leaks}).

Recall that Question~\ref{q:remove-leak} asked whether removing a leak from an identifiable model always yields an identifiable model.  
That question is equivalent to 
asking whether adding a leak preserves unidentifiability.
We conjecture an affirmative answer, as follows.
\begin{conjecture}[Add leak] \label{conj:uniden-add-1-leak}
Let $\mathcal{M}$ be an unidentifiable linear compartmental model. If one leak is added, the resulting model $\widetilde{\mathcal{M}}$ is also unidentifiable.
\end{conjecture}

For strongly connected models, two subcases of Conjecture~\ref{conj:uniden-add-1-leak} naturally 
arise (through Proposition~\ref{prop:MSE}).  
In the first subcase, a model $\mathcal{M}$ is unidentifiable because all maximal minors of the Jacobian matrix are 0.  In the second subcase, unidentifiability arises because the Jacobian matrix has more columns than rows (that is, the model has more parameters than coefficients of the input-output equations), and so the rank of the $m~\times~(|E|+|Leak|)$ Jacobian can not equal $|E|+|Leak|$. 
This second subcase is addressed in the following result.

\begin{theorem}[Add leak] \label{thm:add-leak}
Let $\mathcal{M}$ be a strongly connected linear compartmental model with $|In|=|Out|=1$
and coefficient map $c : \mathbb{R}^{|E| + |Leak|} \to \mathbb{R}^m$.
Consider the model $\widetilde{\mathcal{M}}$ formed by adding a single leak to $\mathcal{M}$. 
If $|E|~+~|Leak|~>~m$ and so $\mathcal{M}$ is unidentifiable, then $\widetilde{\mathcal{M}}$ is also unidentifiable.
\end{theorem}

\begin{proof}
Assume that $|E| + |Leak|$ (the number of parameters of $\mathcal{M}$) is strictly larger than $m$ (the number of coefficients).  For $\widetilde{\mathcal{M}}$, the number of parameters is $|E| + |Leak|+1$, while the number of coefficients is $m$ (if $Leak$ is nonempty) or $m+1$ (if $Leak$ is empty) by 
Proposition~\ref{prop:number-of-coefficients}. 
Thus, by Proposition~\ref{prop:MSE},  $\widetilde{\mathcal{M}}$ is unidentifiable.
\end{proof}

\begin{remark} \label{rem:related-result}
A recent result of Bortner and Meshkat is closely related to Theorem~\ref{thm:add-leak} and pertains to when $|Leak|=|In \cup Out|$~\cite[Theorem 4.1]{bortner-meshkat}.
\end{remark}

%

Next, we consider adding leaks to a model that may or may not be identifiable.  How many leaks are too many, that is, how many leaks need to be added so that the model is automatically unidentifiable?  The following result addresses this question. 

\begin{theorem}[Too many leaks] \label{thm:too-many-leaks}
Let $\mathcal{M}= (G, In, Out, Leak)$ be a 
strongly connected linear compartmental model with at least one leak and $|In|=|Out|=1$.  
Let $n$ be the number of compartments, and let
$L$ be the length of the shortest (directed) path in $G$ from the input compartment to the output compartment. 
If one of the following holds:
\begin{enumerate}
    \item[(I)] $In=Out$ and $|Leak| > \min \{ 1,~ 2n- |E| - 1 \}$, or
    \item[(II)] $In \neq Out$ and $|Leak| > \min \{2,~  2n- |E| - L \}$, 
\end{enumerate}
then $\mathcal{M}$ is unidentifiable.
\end{theorem}
\begin{proof}
If $In=Out$ and $|Leak|>1$, or if $In \neq Out$ and $|Leak|>2$, then Proposition~\ref{prop:bortner-meshkat-leaks} implies that $\mathcal M$ is unidentifiable.  

Now assume that we are in the remaining cases, when 
$In=Out$ and $|Leak| > 2n- |E| - 1 $, or when $In \neq Out$ and $|Leak| >  2n- |E| - L $.   
The model $\mathcal{M}$ has $|E|+|Leak|$ parameters.  
By Proposition~\ref{prop:number-of-coefficients},
the number of 
nontrivial 
coefficients on the left-hand side of
the input-output equation~\eqref{eq:i-o-equation} is
$n$, while the number on the right-hand side is $n-1$ (if $In=Out$) or
$n-L$ (if $In \neq Out$).  
It follows easily from our hypotheses on the number of leaks that $|E|+ |Leak|$ exceeds the number of
coefficients, and so, by Proposition~\ref{prop:MSE}, the model $\mathcal{M}$ is unidentifiable.
\end{proof}

\begin{example} \label{ex:cycle}
Let $\mathcal{M}= (G, In, Out, Leak)$ be a model for which $G$ is the following cycle graph:
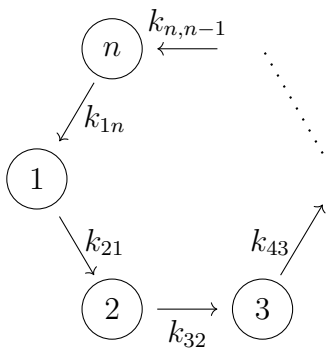
\begin{center}
    \begin{tikzpicture}[scale=2]
  	\draw (-1,0) circle (0.2);	
  	\draw (-0.5,0.866) circle (0.2);	
  	\draw (0.5,-0.866) circle (0.2);	
  	\draw (-0.5,-0.866) circle (0.2);	
    \draw[->] (-0.85, -0.25) -- (-0.62, -0.64 );
  	 \draw[->]  (-0.2,-0.866) -- (0.2,-0.866) ;
 	 \draw[->]  (0.62, -0.64) -- (0.9,-0.17); 
 	 \draw[loosely dotted,thick] (0.9,0.17) -- (0.5,0.866);
 	 \draw[->]  (0.2,0.866) -- (-0.2,0.866);
 	 \draw[->] (-0.62, 0.64 ) -- (-0.85, 0.25)  ;
   	 \node[] at (0.0, -1.03) {$k_{32}$};
   	 \node[] at (0.0, 1.03) {$k_{n,n-1}$};
   	 \node[] at (-0.55, -0.42) {$k_{21}$};
   	 \node[] at (-0.55, 0.4) {$k_{1n}$};
   	 \node[] at (0.55, -0.42) {$k_{43}$};
     	\node[] at (-1, 0) {1};
     	\node[] at  (-0.5,-0.866) {$2$};
     	\node[] at  (+0.5,-0.866) {$3$};
     	\node[] at (-0.5,0.866) {$n$};
 \end{tikzpicture}
\end{center}
Assume that $In=\{1\}$ and $Out=\{n\}$.  It follows from Theorem~\ref{thm:too-many-leaks} 
that $\mathcal{M}$ is unidentifiable if there is more than one leak 
(we have $|E|=n$ and $L=n-1$). 
This result can also be obtained 
from~\cite[Corollary 3.11]{gerberding2019identifiability}.
\end{example}


\subsection{Removing a leak} \label{sec:remove-leak}
In this section we address 
Conjecture~\ref{conj:rmv-leak}, 
which we recall states that, 
starting from an identifiable, strongly connected model with only one leak, 
removing the leak yields a model that is again identifiable.
Here, 
we prove that 
(under some hypotheses)
Conjecture~\ref{conj:rmv-leak} is equivalent to a new conjecture (Conjecture~\ref{conj:leaks-do-not-divide}), 
which 
states that leak terms do not divide the singular-locus equation 
(Theorem~\ref{thm:equivalence-of-conjectures}).  
It will then follow that Conjecture~\ref{conj:rmv-leak} holds whenever  Conjecture~\ref{conj:leaks-do-not-divide} is true (Remark~\ref{rem:when-conjs-true}).

\begin{conjecture}[Leaks are not dividing terms] \label{conj:leaks-do-not-divide}
Let 
 $\widetilde{\mathcal{M}} = (G, In, Out, Leak)$ be a strongly connected linear compartmental model with 
 at least one input and 
 at least one leak.
Assume that $ \widetilde{\mathcal{M}}$ is generically locally identifiable, and has singular-locus equation $f$.  Then $k_{0\ell} \nmid f$ for all $\ell \in Leak$.
\end{conjecture}

The following definition captures when a model has the same number of parameters as coefficients of the input-output equation, that is, the resulting Jacobian matrix is square. 
\begin{definition} \label{def:square-jac}
Let $\mathcal{M} = (G,In,Out,Leak)$ be a strongly connected linear compartmental model with $|In|=|Out|=1$.  The model $\mathcal{M}$ is {\em square-Jacobian} if the sum $|E|+|Leak|$ equals the total number of nontrivial left-hand side and right-hand side coefficients asserted in Proposition~\ref{prop:number-of-coefficients}. 
\end{definition}

\begin{theorem}[Equivalence of Conjectures~\ref{conj:rmv-leak} and~\ref{conj:leaks-do-not-divide}] \label{thm:equivalence-of-conjectures}
Let $\widetilde{\mathcal{M}} = (G,In,Out,Leak)$ be a strongly connected linear compartmental model 
with $|In|=|Out|=|Leak|=1$.
Assume that $ \widetilde{\mathcal{M}}$ is generically locally identifiable and square-Jacobian. 
Then, Conjecture~\ref{conj:rmv-leak} holds for 
$\widetilde{\mathcal{M}}$
if and only if 
Conjecture~\ref{conj:leaks-do-not-divide} holds for 
$\widetilde{\mathcal{M}}$.
\end{theorem}

\begin{proof}
Let $\ell$ denote the unique leak compartment. 
By Proposition~\ref{prop:number-of-coefficients}, and the square-Jacobian hypothesis, 
the number of 
nontrivial 
coefficients of the input-output equation~\eqref{eq:i-o-equation} for $\widetilde{\mathcal{M}}$ is $r=|E|+ |Leak|$.  We denote these coefficients by $\widetilde{c}_1, \widetilde{c}_2, \dots , \widetilde{c}_r$, where we have chosen some ordering so  that $\widetilde{c}_r$ is the constant term of the left-hand side of~\eqref{eq:i-o-equation}.

By Lemma~\ref{lem:coeff-add-leak}, 
the corresponding coefficients of the input-output equation for the model $\mathcal{M} $
(obtained by removing the leak from $\mathcal M'$)
are 
\begin{align} \label{eq:coeffs-related}
c_i = \widetilde{c}_i|_{k_{0 \ell}=0}~,
\end{align}
for $i=1,2,\dots, r$.  
We also know that $c_r=0$ (see~\cite[Remark~2.1]{gerberding2019identifiability}).

Let $J$ be the Jacobian matrix of $(c_1,c_2, \dots, c_{r-1})$, with respect to some ordering of the $|E|$ parameters of $\mathcal{M}$.  Let $\widetilde{J}$ denote the Jacobian matrix of $(\widetilde{c}_1,\widetilde{c}_2, \dots, \widetilde{c}_{r-1})$ with respect to the same ordering of parameters, plus (for the last column of the matrix) the leak parameter $k_{0 \ell}$.  We claim that the matrices $J$ and $\widetilde J$ 
(which are square by the square-Jacobian hypothesis) 
are related as follows:
	\begin{align} \label{eq:compare-two-jacobians}
	\widetilde J |_{k_{0 \ell }=0} 
	~&=~
	 \left( \begin{array}{ccc|c}
	~& ~& ~ & * \\
	~& J & ~ & \vdots \\
	~& ~& ~ & * \\
\hline
	0 & \dots & 0 & \frac{\partial  \widetilde c_r }{\partial k_{0 \ell}} \\
 \end{array} \right)
	\end{align}

To prove the claim, we begin by noting that equation~\eqref{eq:coeffs-related} and the equality $c_r=0$ imply 
that the matrix on the left-hand side of~\eqref{eq:compare-two-jacobians} 
and the matrix on the right-hand side 
have the same columns, except possibly the last column.  Next, the entry in the lower-right corner of the matrix $\widetilde{J}|_{k_{0 \ell}=0}$ is 
$\frac{\partial  \widetilde c_r }{\partial k_{0 \ell}} |_{k_{0 \ell}=0}  = 
\frac{\partial  \widetilde c_r }{\partial k_{0 \ell}} $, where we are using the fact that $k_{0 \ell}$ appears only linearly (and not to higher powers) in $\widetilde{c}_r$.  Hence, the claimed equality~\eqref{eq:compare-two-jacobians} holds, and so we can take determinants to obtain:

\begin{align} \label{eq:2-dets}
    \left( \det \widetilde J \right)|_{k_{0 \ell}=0} ~=~
        \frac{\partial  \widetilde c_r }{\partial k_{0 \ell}} \cdot  ( \det J ) ~.
    \end{align}

We next claim that $\frac{\partial  \widetilde c_r }{\partial k_{0 \ell}} $ is nonzero. Indeed, this must hold in order for the inequality 
$\widetilde{c}_r \neq 0$ and the equalities
$0 = c_r = \widetilde{c}_r|_{k_{0 \ell} = 0}$ 
to hold.

Conjecture~\ref{conj:leaks-do-not-divide} holds for 
$\widetilde{\mathcal{M}}$ if and only if the left-hand side of~\eqref{eq:2-dets} is nonzero.  Also, 
Conjecture~\ref{conj:rmv-leak} holds for 
$\widetilde{\mathcal{M}}$
if and only if $\det J$ (appearing in the right-hand side of~\eqref{eq:2-dets}) is nonzero.  Thus, because $\frac{\partial  \widetilde c_r }{\partial k_{0 \ell}} \neq 0$, we see that the two conjectures are equivalent for $\widetilde{\mathcal M}$.
\end{proof}

\begin{remark}[Square vs.\ non-square Jacobian]
Theorem~\ref{thm:equivalence-of-conjectures} pertains to square-Jacobian models.  
For non-square-Jacobian models that are identifiable, there are more coefficients of the input-output equation~\eqref{eq:i-o-equation} than parameters, and being identifiable means that at least one minor of the Jacobian matrix is nonzero. If such a minor comes from a submatrix containing the row corresponding to the constant term of the left-hand side of the input-output equation (i.e., the coefficient $\widetilde{c}_r$ in the proof of Theorem~\ref{thm:equivalence-of-conjectures}), then the proof easily generalizes to accommodate this case.  
On the other hand, when no such minor exists, we do not know how to address this scenario.  However, we have never observed such a model!
\end{remark}

\begin{remark} \label{rem:when-conjs-true}
By Theorem~\ref{thm:equivalence-of-conjectures}, Conjecture~\ref{conj:leaks-do-not-divide} is true whenever Conjecture~\ref{conj:rmv-leak} holds, for instance, in the cases listed after Conjecture~\ref{conj:rmv-leak} (which include certain catenary, cycle, and mammillary models).
\end{remark}

\section{Results on dividing edges} \label{sec:edge}
In this section, we address Conjecture~\ref{conj:dividing-edge}, 
which we recall states that removing a dividing edge from an identifiable model results in a model that, if strongly connected, is unidentifiable.  Here we prove a special case of this conjecture, which pertains to the square-Jacobian case (Theorem~\ref{thm:dividing-edge}), and then investigate the  non-square-Jacobian case (Section~\ref{sec:non-square}).

\subsection{When distance from input to output increases}

We have observed that, for some models, a dividing edge 
appears 
in the shortest path from the input compartment to the output compartment and its removal increases the length of such a path. We can therefore apply the conjectured formula for the number of coefficients to obtain the following result.


\begin{theorem}[Remove dividing edge] \label{thm:dividing-edge}
Let $\mathcal{M}$ be a strongly connected linear compartmental model with $|In| = |Out|=1$ and $Leak = \emptyset$.
Assume that $\mathcal M$ is square-Jacobian and generically locally identifiable.  Assume, moreover, that $k_{ij}$ is a dividing edge such that the model
 $\mathcal{M}^\prime$ obtained from $\mathcal M$ by removing the edge $k_{ij}$ is strongly connected and the length of the shortest path from input to output has increased by at least $2$. 
Then $\mathcal M '$ is unidentifiable.
\end{theorem}

\begin{proof}
We know that $In \neq Out$, as the length of the shortest path from input to output increases when $k_{ij}$ is removed. For $\mathcal{M}$, the number of coefficients equals the number of parameters, namely, $|E|+|Leak|$.  As for ${\mathcal{M}}'$, the number of parameters is $|E|+|Leak| -1$, and (by Proposition~\ref{prop:number-of-coefficients}) the number of coefficients has decreased by at least 2 and so is at most $|E|+|Leak| -2$.  Thus, as there are more parameters than coefficients of the input-output equation, $\mathcal{M}'$ is unidentifiable (by Proposition~\ref{prop:MSE}).
\end{proof}


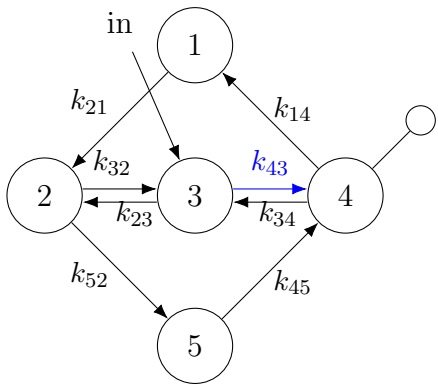
\begin{figure}[ht]
    \centering
    \begin{tikzpicture}
        \node[circle,draw, minimum size=1cm] (2) at  (0,0) {2}; 
        \node[circle,draw, minimum size=1cm] (3) at  (2,0)  {3};
        \node[circle,draw, minimum size=1cm] (4) at  (4,0)  {4};
        \node[circle,draw, minimum size=1cm] (5) at  (2,-2)  {5};
        \node[circle,draw, minimum size=1cm] (1) at  (2, 2)  {1};
        \draw (1) -- (2) [-{Latex[length=2mm]}] node [label={[xshift= .6cm, yshift= .8cm]{\small $k_{21}$}}]{};
        \draw (3) -- (2) [-{Latex[length=2mm]},transform canvas={yshift=-2.5pt}] node [label={[xshift= 1.2cm, yshift= -.6cm]{\small $k_{23}$}}]{};
        \draw (2) -- (3) [-{Latex[length=2mm]},transform canvas={yshift=2.5pt}] node [label={[xshift= -1.1cm, yshift= -.1cm]{\small $k_{32}$}}]{};
        \draw (3) -- (4) [-{Latex[length=2mm]},transform canvas={yshift=2.5pt}, color = blue] node [label={[xshift= -1cm, yshift= -.1cm, color = black]{\small {\color{blue} $k_{43}$} }}]{};
        \draw (4) -- (3) [-{Latex[length=2mm]},transform canvas={yshift=-2.5pt}] node [label={[xshift= 1.1cm, yshift= -.6cm]{\small $k_{34}$}}]{};
        \draw (4) -- (1) [-{Latex[length=2mm]}] node [label={[xshift= 1.3cm, yshift= -1.3cm]{\small $k_{14}$}}]{};
        \draw (5) -- (4) [-{Latex[length=2mm]}] node [color = black, label={[xshift= -.7cm, yshift= -1.6cm]{\small $k_{45}$}}]{};
        \draw (2) -- (5) [-{Latex[length=2mm]}] node [color = black, label={[xshift= -1.4cm, yshift= .5cm]{\small $k_{52}$}}]{};
        \node[circle, minimum size=0cm] (in) at  (1,2.3)  {in};
        \draw (in) -- (3) [-{Latex[length=2mm]}];
        \node[circle, draw, minimum size=.01cm] (out) at  (5,1)  {};
        \draw (out) -- (4);
    \end{tikzpicture}
    \caption{Model with $In = \{3\}$, $Out = \{4\}$, and dividing edge $k_{43}$ (in blue).}
    \label{fig:not_break_strong2}
\end{figure}

\begin{example} \label{ex:dividing-edge-use-thm}
Figure \ref{fig:not_break_strong2} depicts an identifiable, strongly connected, square-Jacobian linear compartmental model $\mathcal M$ with dividing edge $k_{43}$. The model has $8$ parameters, and the shortest path from input to output has length $1$ (the edge $k_{43}$). 
When the dividing edge $k_{43}$ is removed, the resulting model $\mathcal M '$ is also strongly connected, but the length of the shortest path from input to output has increased to $3$ (the edges of this path are $k_{23}, \, k_{52}, \, k_{45}$). 
Thus, by Theorem~\ref{thm:dividing-edge}, $\mathcal M' $ is unidentifiable.
\end{example}

The model in the following example is not covered by Theorem~\ref{thm:dividing-edge}.

\begin{example} \label{ex:mysterious1}
We revisit the linear compartmental model from Figure \ref{fig:Mysterious1} (in Section~\ref{sec: background}), which is strongly connected and square-Jacobian. The dividing edges are $k_{41}$, $k_{43}$, and
$k_{54}$. 
The edge $k_{41}$ is an interesting case of a dividing edge, in that the removal of $k_{41}$ results in a model $\mathcal M'$ that is still strongly connected 
and the length of the shortest path from input to output has not increased (so, Theorem~\ref{thm:dividing-edge} does not apply). 
Nevertheless, 
consistent with Conjecture~\ref{conj:dividing-edge}, 
this model $\mathcal M'$ is unidentifiable: there are $8$ coefficients and $7$ parameters, and all minors of the $8 \times 7$ Jacobian matrix are 0. 
As for the other dividing edges, $k_{43}$ and $k_{54}$, removing either one yields a model that is not strongly connected and so is not considered by Conjecture~\ref{conj:dividing-edge}. 
\end{example}

We end this subsection by returning to a model from Section~\ref{sec: background}.

\begin{example} \label{ex:appearing-1}
Recall that the linear compartmental model shown in Figure~\ref{fig:appearing1} is strongly connected
and square-Jacobian ($6 \times 6$), with 2 dividing edges: $k_{12}$ and $k_{23}$.  
Removing the edge $k_{23}$ yields a model that is no longer strongly connected, so Conjecture~\ref{conj:dividing-edge} does not apply.  
Removing the edge~$k_{12}$ yields a strongly connected model that is unidentifiable, which is 
consistent with Conjecture~\ref{conj:dividing-edge}.  
This unidentifiability is explained by Theorem~\ref{thm:dividing-edge}: the length of the shortest path from input to output increases from 1 to 3.
\end{example}
\subsection{Non-square Jacobian matrices} \label{sec:non-square}
We end this section 
by briefly considering the case when an identifiable model has a non-square Jacobian matrix and therefore the singular locus is defined by the set of all maximal minors of the Jacobian matrix. 
In this case, we propose to extend Conjecture~\ref{conj:dividing-edge} as follows:
If an edge divides every maximal minor of the Jacobian matrix, then removing this edge results in a model that is unidentifiable.

It is also natural to consider the possible, stronger conjecture that if an edge divides {\em at least one} maximal minor of the Jacobian matrix, then removing this edge makes the model unidentifiable. 
However, this conjecture is false, as the next example shows.  
Indeed, more work remains to be done in 
order to understand the information contained -- both collectively and individually --
in the minors of the Jacobian matrix.

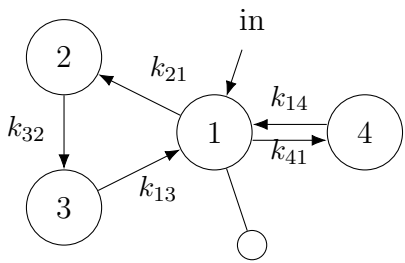
\begin{figure}[ht]
    \centering
    \begin{tikzpicture}
        \node[circle,draw, minimum size=1cm] (3) at  (0,-1)  {3};
        \node[circle,draw, minimum size=1cm] (2) at  (0,1)  {2};
        \node[circle,draw, minimum size=1cm] (1) at  (2,0)  {1};
        \node[circle,draw, minimum size=1cm] (4) at  (4,0)  {4};
        \draw (1) -- (2) [-{Latex[length=2mm]}] node [label={[xshift= 1.4cm, yshift= -.6cm]{\small $k_{21}$}}]{};
        \draw (2) -- (3) [-{Latex[length=2mm]}] node [label={[xshift= -.5cm, yshift= .6cm]{\small $k_{32}$}}]{};
        \draw (1) -- (4) [-{Latex[length=2mm]},transform canvas={yshift=-3pt}] node [label={[xshift= -1cm, yshift= -.6cm]{\small $k_{41}$}}]{};
        \draw (4) -- (1) [-{Latex[length=2mm]},transform canvas={yshift=3pt}] node [label={[xshift= 1cm, yshift= -.1cm]{\small $k_{14}$}}]{};
        \draw (3) -- (1) [-{Latex[length=2mm]}] node [label={[xshift= -.75cm, yshift= -1.2cm]{\small $k_{13}$}}]{};
        
        \node[circle, minimum size=0cm] (in) at  (2.5,1.5)  {in};
        \draw (in) -- (1) [-{Latex[length=2mm]}];
        \node[circle, draw, minimum size=.01cm] (out) at  (2.5,-1.5)  {};
        \draw (out) -- (1);
    \end{tikzpicture}
    \caption{Model with $In = Out= \{1\}$.}
    \label{fig:one_det_edge}
\end{figure}

\begin{example} \label{ex:divide-1-vs-all-minors}
Depicted in Figure~\ref{fig:one_det_edge} is a strongly connected linear compartmental model with $5$ parameters and $6$ coefficients, 
which results in a $6 \times 5$ Jacobian matrix of the coefficient map. 
 The 6 minors of this matrix are shown here in factored form:  
 \begin{align*}
M_1 \quad & = \quad     
    -(k_{13} k_{14}^2 - k_{14}^3 - k_{13} k_{14} k_{21} + k_{14}^2 k_{21} - k_{13} k_{14} k_{32} + k_{14}^2 k_{32} - k_{14} k_{21} k_{32} - k_{14}^2 k_{41} 
    \\
    & \quad \quad \quad \quad 
    + k_{13} k_{32} k_{41}) (k_{13} - k_{32})~,
    \\
%
%
   M_2 \quad & = \quad     
    -(k_{13}^2 k_{14}^2 - k_{13} k_{14}^3 - k_{13} k_{14}^2 k_{21} + k_{14}^3 k_{21} - 2 k_{13}^2 k_{14} k_{32} + 3 k_{13} k_{14}^2 k_{32} - k_{14}^3 k_{32} 
     \\
    & \quad \quad \quad \quad 
    - k_{14}^2 k_{21} k_{32} + k_{13}^2 k_{32}^2 - 2 k_{13} k_{14} k_{32}^2 + k_{14}^2 k_{32}^2 - k_{13} k_{14}^2 k_{41} + 2 k_{13} k_{14} k_{32} k_{41} 
        \\ \displaybreak[0]
    & \quad \quad \quad \quad 
    - k_{14}^2 k_{32} k_{41}) (k_{13} - k_{32})~,
    \\
   M_3 \quad & = \quad     
    -(k_{13}^3 k_{14}^2 - k_{13}^2 k_{14}^3 - k_{13}^2 k_{14}^2 k_{21} + k_{13} k_{14}^3 k_{21} - 2 k_{13}^3 k_{14} k_{32} + 3 k_{13}^2 k_{14}^2 k_{32} - k_{13} k_{14}^3 k_{32}
            \\
    & \quad \quad \quad \quad 
    + 2 k_{13}^2 k_{14} k_{21} k_{32} - 4 k_{13} k_{14}^2 k_{21} k_{32} + k_{14}^3 k_{21} k_{32} + k_{13}^3 k_{32}^2 - 3 k_{13}^2 k_{14} k_{32}^2 
        \\
    & \quad \quad \quad \quad 
    + 3 k_{13} k_{14}^2 k_{32}^2 - k_{14}^3 k_{32}^2 - k_{13}^2 k_{21} k_{32}^2 + 2 k_{13} k_{14} k_{21} k_{32}^2 - k_{14}^2 k_{21} k_{32}^2 + k_{13}^2 k_{32}^3 
        \\
    & \quad \quad \quad \quad 
    - 2 k_{13} k_{14} k_{32}^3 + k_{14}^2 k_{32}^3 - k_{13}^2 k_{14}^2 k_{41} + 2 k_{13}^2 k_{14} k_{32} k_{41} - k_{13} k_{14}^2 k_{32} k_{41} - k_{13}^2 k_{32}^2 k_{41} 
        \\
    & \quad \quad \quad \quad 
    + 2 k_{13} k_{14} k_{32}^2 k_{41} - k_{14}^2 k_{32}^2 k_{41}) (k_{13} - k_{32})~,
  \\ \displaybreak[0]
   M_4 \quad & = \quad     
-(k_{13} - k_{14}) (k_{13} - k_{32}) (k_{14} - k_{32}) k_{14}~,
  \\
   M_5 \quad & = \quad     
-(k_{13} k_{14} - k_{13} k_{32} + k_{14} k_{32}) (k_{13} - k_{14}) (k_{13} - k_{32}) (k_{14} - k_{32})~,
  \\
   M_6 \quad & = \quad     
-(k_{13}^2 k_{14} - k_{13}^2 k_{32} + k_{13} k_{14} k_{32} - k_{13} k_{32}^2 + k_{14} k_{32}^2) (k_{13} - k_{14}) (k_{13} - k_{32}) (k_{14} - k_{32})~.
  \end{align*}
Notice that no edge parameter $k_{ij}$ divides all $6$ of these minors. However, the edge $k_{14}$ divides one of the  $6$ minors, $M_4$.
When this edge is removed, the resulting model $\mathcal{M}'$ is no longer strongly connected, but by thinking of the edge $k_{41}$ as a ``pseudo-leak'', we can nevertheless view the model $\mathcal{M}'$ as a (strongly connected) cycle model on three nodes with input, output, and leak in a single compartment -- and so $\mathcal{M}'$ is identifiable~\cite{GNA17Three}.
\end{example}

\section{Discussion} \label{sec:discussion}
This article was motivated by the following questions about linear compartmental models:
\begin{question} \label{q:discussion}~
\begin{enumerate}
    \item (When) does adding or removing an edge or a leak preserve identifiability?
    \item (When) do edge or leak terms divide the singular-locus equation?
    \item How do the above two questions interact?
\end{enumerate}
\end{question}

Some conjectured (partial) answers are as follows. 
Removing a leak preserves identifiability (Conjecture~\ref{conj:rmv-leak}), 
removing a dividing edge never preserves identifiability (Conjecture~\ref{conj:dividing-edge}), 
and leak terms do not divide the singular-locus equation (Conjecture~\ref{conj:leaks-do-not-divide}).  We proved several subcases of these conjectures 
and proved (under some hypotheses) the equivalence of  Conjectures~\ref{conj:rmv-leak} and~\ref{conj:leaks-do-not-divide}.

Going forward, there remain 
many cases when a model is unidentifiable but nevertheless has at least as many coefficients as parameters.  This is an interesting future direction.
More generally, we hope that our results inspire more answers to Question~\ref{q:discussion}, which in turn will contribute to our ability to read important information about a model directly from its structure.  Indeed, it would be spectacular to be able to infer immediately from the underlying combinatorics of a model which edges are dividing edges or even whether or not the model is identifiable.

\subsection*{Acknowledgements}
Patrick Chan, Katherine Johnston, and Clare Spinner initiated this research in the 2020 REU in the Department of Mathematics at Texas A\&M
University, supported by NSF grant DMS-1757872, in which Anne Shiu and Aleksandra Sobieska were mentors.  
Anne Shiu was supported by NSF grant DMS-1752672, and acknowledges Cashous Bortner and Nicolette Meshkat for helpful discussions.

\bibliography{mybibliography}{}

\begin{thebibliography}{10}

\bibitem{bearup}
Daniel~J. Bearup, Neil~D. Evans, and Michael~J. Chappell.
\newblock The input–output relationship approach to structural
  identifiability analysis.
\newblock {\em Comput.\ Meth.\ Prog.\ Bio.}, 109(2):171--181, 2013.

\bibitem{Bellman}
R.~Bellman and K.J. {\AA}str{\"o}m.
\newblock On structural identifiability.
\newblock {\em Math.\ Biosci.}, 7(3--4):329 -- 339, 1970.

\bibitem{bortner}
Cashous Bortner, Elizabeth Gross, Nicolette Meshkat, Anne Shiu, and Seth
  Sullivant.
\newblock Identifiability of linear compartmental tree models.
\newblock Available from {\tt arXiv:2106.08487}, 2021.

\bibitem{bortner-meshkat}
Cashous Bortner and Nicolette Meshkat.
\newblock Identifiable paths and cycles in linear compartmental models.
\newblock {\em Available from {\tt arXiv:2010.07203}}, 2020.

\bibitem{dargenio1988simulation}
David~Z D'Argenio, Alan Schumitzky, and Walter Wolf.
\newblock Simulation of linear compartment models with application to nuclear
  medicine kinetic modeling.
\newblock {\em Comput.\ Meth.\ Prog.\ Bio.}, 27(1):47--54, 1988.

\bibitem{gerberding2019identifiability}
Seth Gerberding, Nida Obatake, and Anne Shiu.
\newblock Identifiability of linear compartmental models: The effect of moving
  inputs, outputs, and leaks.
\newblock {\em Linear and Multilinear Algebra, to appear}, 2020.

\bibitem{glad}
S.~T. Glad.
\newblock {\em Differential Algebraic Modelling of Nonlinear Systems}, pages
  97--105.
\newblock Birkh{\"a}user Boston, Boston, MA, 1990.

\bibitem{GHMS19One}
Elizabeth Gross, Heather~A. Harrington, Nicolette Meshkat, and Anne Shiu.
\newblock Linear compartmental models: input-output equations and operations
  that preserve identifiability.
\newblock {\em SIAM J.\ Appl.\ Math.}, 79(4):1423--1447, 2019.

\bibitem{joining}
Elizabeth Gross, Heather~A Harrington, Nicolette Meshkat, and Anne Shiu.
\newblock Joining and decomposing reaction networks.
\newblock {\em J.\ Math.\ Biol.}, 80:1683--1731, 2020.

\bibitem{GNA17Three}
Elizabeth Gross, Nicolette Meshkat, and Anne Shiu.
\newblock Identifiability of linear compartment models: the singular locus.
\newblock {\em Preprint, {\tt arXiv:1709.10013}}, 2017.

\bibitem{hori2006role}
Sharon~S Hori, Irwin~J Kurland, and Joseph~J DiStefano.
\newblock Role of endosomal trafficking dynamics on the regulation of hepatic
  insulin receptor activity: {M}odels for {Fao} cells.
\newblock {\em Ann.\ Biomed.\ Eng.}, 34(5):879, 2006.

\bibitem{MSE-iden-results-several}
Nicolette Meshkat, Seth Sullivant, and Marisa Eisenberg.
\newblock Identifiability results for several classes of linear compartment
  models.
\newblock {\em B.\ Math.\ Biol.}, 77(8):1620--1651, 2015.

\bibitem{mulholland1974analysis}
Robert~J Mulholland and Marvin~S Keener.
\newblock Analysis of linear compartment models for ecosystems.
\newblock {\em J.\ Theor.\ Biol.}, 44(1):105--116, 1974.

\bibitem{Ovchinnikov-Pogudin-Thompson}
Alexey Ovchinnikov, Gleb Pogudin, and Peter Thompson.
\newblock Input-output equations and identifiabilty of linear {ODE} models.
\newblock {\em Preprint, {\tt arXiv:1910.03960}}, 2019.

\bibitem{vajda1984}
S.~Vajda.
\newblock Analysis of unique structural identifiability via submodels.
\newblock {\em Math. Biosci.}, 71:125--146, 1984.

\bibitem{vajdaetal}
S.~Vajda, J.~J. DiStefano, III, K.~R. Godfrey, and J.~Fagarasan.
\newblock Parameter space boundaries for unidentifiable compartmental models.
\newblock {\em Math. Biosci.}, 97:27--60, 1989.

\end{thebibliography}
\bibliographystyle{plain}

\end{document}